\documentclass[11pt]{amsart}

\usepackage[utf8]{inputenc} 
\usepackage[english]{babel}
\usepackage[T1]{fontenc}

\usepackage{url}

\usepackage{amsmath}
\usepackage{amsfonts}
\usepackage{amssymb}
\usepackage{pst-node}
\usepackage{tikz-cd}
\usepackage{tikz}
\usetikzlibrary{trees}
\usetikzlibrary{arrows}
\usepackage{amsthm}
\usepackage{mathrsfs}
\usepackage{bm}
\usepackage{mathtools}
\usepackage[framemethod=TikZ]{mdframed}

\usepackage{units}

\usepackage{listings}
\usepackage[margin = 60pt]{geometry}

\usepackage[babel]{csquotes}
\usepackage{hyperref}
\usepackage[
backend=biber,
style=numeric,
sorting=nty,
maxnames=10
]{biblatex}
\input xy
\xyoption{all}

\theoremstyle{theorem}
\newtheorem{thm}{Theorem}[section]

\newtheorem{prp}[thm]{Proposition}
\newtheorem{crl}[thm]{Corollary}
\newtheorem{lemma}[thm]{Lemma}

\theoremstyle{definition}
\newtheorem{df}[thm]{Definition}
\newtheorem{ex}[thm]{Example}

\theoremstyle{remark}

\newtheorem{rmk}[thm]{Remark}

\DeclareMathOperator{\CHilb}{CHilb}
\DeclareMathOperator{\Hilb}{Hilb}
\DeclareMathOperator{\di}{dim}
\newcommand{\dimm}{ \, \di }
\newcommand{\C}{\mathbb{K}}
\newcommand{\Li}{\mathbb{L}}

\begin{document}

\title{Motivic classes of curvilinear Hilbert schemes and Igusa zeta functions}
\author{Ilaria Rossinelli}
\address{Department of Mathematics, EPFL Lausanne}
\email{ilaria.rossinelli@epfl.ch}
\date{}
\maketitle

\begin{abstract}
This paper delves into the study of curvilinear Hilbert schemes associated with a singular variety $(X,0)$ and the relationship between their motivic classes and the motivic measure on the arc scheme $X_\infty$ of $X$ introduced by Denef and Loeser.

We introduce an Igusa zeta function specifically tailored for curvilinear Hilbert schemes for which we provide an explicit formulation in terms of an embedded resolution of the singularity, and we consequently obtain a recursive formula to compute the motivic classes of curvilinear Hilbert schemes in terms of the resolution.

In addition, the paper explores and analyzes the geometry and combinatorics of curvilinear Hilbert schemes in the context of plane curve singularities and their topological invariants.
\end{abstract}

\tableofcontents

\section{Introduction} \label{intro}
The motivic measure $\mu$ was introduced by Kontsevich \cite{kon} as a measure on the arc scheme of a smooth variety, and Denef and Loeser later extended it to any singular variety in \cite{mea} and \cite{DL}.
For a given variety $X$ over an algebraically closed field $\C$ of characteristic zero, its arc scheme is defined as the scheme $X_\infty$ characterized by the property that for any field extension $K$ of $\C$ the set of $K$-points $X_\infty (K)$ is in bijection with the set $$\operatorname{Hom}_{\operatorname{Sch}_\C}(\operatorname{Spec}(K[[t]]) , X)$$ of $K[[t]]$-points of the variety $X$. 
Motivic integrals are then defined as integrals of suitable functions on the arc scheme of any given variety with respect to this measure it is equipped with. 

Our second object of interest is the punctual Hilbert scheme $\Hilb^k_0(X)$ of $X$ at a point $0 \in X$, defined as the variety parameterizing zero-dimensional subschemes of $X$ of colength $k$ entirely supported at $0$, for which the historical reference is \cite{EGA}. Roughly speaking, it classifies fat points of multiplicity $k$ on $X$. In particular, given a singular variety $(X,0)$ we are interested in studying the singularity through the punctual Hilbert scheme of zero-dimensional subschemes supported at the singularity.

In this paper we focus on curvilinear subschemes, appearing for instance in \cite{Berczi},
as they are naturally related to (truncations of) the arc scheme of the affine space. Indeed, the curvilinear Hilbert scheme of $X$ is defined as the open set of points $Z \in \Hilb^k_0(X)$ such that $$\mathcal{O}_{Z}\cong \C[t] / (t^k).$$

We aim to further develop the theory of motivic integration \textit{in relationship to} Hilbert schemes, and specifically express some motivic integrals in terms of motivic classes of curvilinear Hilbert schemes.
This is accomplished in the following theorem presented in Section \ref{curv}, which is formulated in the general setting where $(W,0)$ is a smooth variety of dimension $n$ and $V \subseteq W$ is a closed subvariety. It relies on the construction of a geometric morphism $$\Phi: Q_0^k(V) \to \CHilb^k_0(V)$$ relating a certain quotient variety, denoted by $Q^k_0(X)$, of the (truncation of) points $\gamma \in W_\infty$ such that $ord_0(\gamma) = 1$ and $ord_V(\gamma) = k$ to curvilinear zero-dimensional subschemes of length $k$ of $V$ supported at $0$, that is the points of $\CHilb^k_0(X)$. 
\begin{thm} \label{main1}
Let $(W,0)$ be a smooth variety of dimension $n$ and let $V \subseteq W$ be a closed subvariety with $0\in V$.
For $k >0$ let $\pi_k : W_\infty \to W_k$ be the morphism of schemes induced by the projection $\C[[t]] \to \C[t]/(t^k)$ between the arc scheme $W_\infty$ and the $k$-jet scheme $W_k$ and let $A = \pi_1^{-1}(0) \setminus \pi_2^{-1}(0) \subseteq W_\infty$.
Then, we have $$\int_A (\Li^{-ord_V})^s d\mu = 
\frac{1}{\Li^s}\frac{1}{\Li^n}\left(\frac{1}{\Li^n} - \frac{1}{\Li^{2n}}\right)   +   \frac{\Li -1}{\Li^2} (1 - \Li^s) Q_V( \Li^{-s - (n-1)} )$$ where $$Q_V(T) = \displaystyle\sum_{k\geq 2} [\CHilb^k_0(V)] \, T^k$$ is the generating series of motivic classes of curvilinear Hilbert schemes.
\end{thm}

Moreover, with an eye towards singularity theory, we are interested in and ultimately able to explicitly express the Igusa zeta function $\displaystyle\int_A (\Li^{-ord_V})^s d\mu$ hence $Q_V(T)$ in terms of an embedded resolution of singularities of the pair $(W,V)$. From the construction of $\Phi$ and studying the behavior of the order function $ord_0: W_\infty \to \mathbb N$ with respect to resolution maps we obtain the following theorem, again presented in Section \ref{curv}.
\begin{thm} \label{main2}
Let $\pi: (Y, E) \rightarrow (W,V)$ be an embedded resolution of the pair $(W,V)$ as in Definition \ref{res} and let $\widetilde{V}$ denote the strict transform of $V$ under $\pi$. 
Then, we have
$$ \displaystyle\int_{A} (\Li^{-ord_V})^s d\mu = \left(\frac{1}{\Li} - \frac{1}{\Li^2}\right)\Li^{-(n-1)} \displaystyle\sum_{i \in S} \frac{1}{\Li^{N_i s + \nu_j}} \left(  [(E_i \setminus \widetilde{V})^o]   +   \frac{1}{\Li^{s+1} - 1} [(E_i \cap \widetilde{V})^o](\Li -1)   \right)$$ where $S$ denotes the set of divisors of order $1$ (see Definition \ref{def-order}).
\end{thm}

As one of the main consequences, we obtain a complete description of the motivic classes $[\CHilb^k_0(V)]$ of curvilinear Hilbert schemes in terms of an embedded resolution of singularities.
\begin{crl} \label{main3}
For $k \geq 2$ we have 
$$\Li^{-(n-1)k} [\CHilb^k_0(V)] - \Li^{-(k+1)(n-1)}[\CHilb^{k+1}_0(V)]) = $$
$$
\Li^{-(n-1)} \left( 
\displaystyle\sum_{i \in S \, | \, N_i = k} 
\Li^{-\nu_i}[(E_i \setminus \widetilde{V})^o] \, \,
+ \, \,
\displaystyle\sum_{i \in S , \, m  \geq 1 \, | \, N_i +m = k } 
\Li^{ -\nu_i -m}(\Li - 1)[(E_i \cap \widetilde{V})^o] \right) .$$
\end{crl}

We ultimately specialized our work to the case of (isolated reduced) complex plane curve singularities in Section \ref{ex}. In this case, particular interest lies in the geometric interpretation of topological invariants of the singularity i.e. link invariants. 

In this direction, from Corollary \ref{main3} we prove that the motivic classes of curvilinear Hilbert schemes of complex plane curves are polynomial in $\Li$ and, when working with isolated reduced singularities, we can consider their link and these polynomials become as well link invariants with a precise geometric description in terms of an embedded resolution. 
\begin{crl} \label{main4}
Let $(C, 0)$ be a complex plane curve.
For any $k\geq 2$ we have that $[\CHilb^k_0(C)]$ is polynomial in $\Li$ and assuming $(C,0)$is an isolated reduced complex plane curve singularity, the polynomial $[\CHilb^k_0(C)]$ is a topological invariant of $(C,0)$.
\end{crl}
We also briefly discuss the relationship between the conjecture presented in \cite{ors} and our work developed in Section \ref{curv}, since the mentioned conjecture concerns the interpretation of link invariants in terms of weight polynomials of Hilbert schemes. 

Finally, for plane curve singularities with no smooth branches we obtain a threshold in terms of an embedded resolution starting from which the curvilinear Hilbert schemes of the curve singularity become empty.
\begin{thm} \label{main5}
Let $(C, 0)$ be a complex plane curve singularity with no smooth branches and let $N = \max \{ N_i \, | \, i \in S \}$ for an embedded resolution following the notations of Definition \ref{res} and Definition \ref{def-order}. Then, $\CHilb^k_0(C)$ is empty for any $k>N$ and this bound is optimal.
\end{thm}


\vspace{4.5mm}
This paper is organized as follows. We recall the main notions and results of the theory of motivic integration and Hilbert schemes in Section \ref{back}.
Then, the first part of Section \ref{curv} is dedicated to introducing and studying the motivic class of geometric quotients of the jet scheme of a given variety. Later, we investigate the relationship of these motivic classes with curvilinear Hilbert schemes to prove Theorem \ref{main1}, and finally the behavior of arcs under an embedded resolution of singularities to prove Theorem \ref{main2}.
In Section \ref{ex} we then focus on studying the curvilinear invariants introduced in Section \ref{curv} and some specific properties of curvilinear Hilbert schemes of complex plane curves.


\vspace{4.5mm}
\indent
\textbf{Acknowledgements.} I would like to deeply thank Dimitri Wyss for his advice and supervision during the discussion and writing stages of this paper. Moreover, we would like to thank Alexei Oblomkov for his useful suggestions and insight, as well as Eugene Gorsky for his comments.
This work was supported by the Swiss National Science Foundation [No. 196960].

\section{Background} \label{back}

We denote by $\C$ an algebraically closed field of characteristic zero. We denote by $\operatorname{Sch}_\C$ the category of noetherian schemes over $\C$ and by $\operatorname{Var}_\C$ the category of varieties i.e. separated integral schemes of finite type over the same field $\C$. Throughout this paper, when referring in general to a scheme or a variety we imply objects of these varieties and over $\C$.

\subsection{Motivic integration}

In this section, we recall some basic facts about the theory of motivic integration and in particular of motivic Igusa zeta functions. We refer to \cite{greenbook} or \cite{DL} as references on the subject.

We start by introducing the fundamental notions of $k$-jet scheme and arc scheme, and their truncation maps, which will have a central role in the definition of motivic integrals. 

\begin{df} \label{def-the jet scheme}
Let $X, Y$ be schemes. For any $k > 0$ we denote by $X_k(Y)$ the set 
$$X_k(Y) = \operatorname{Hom}_{\operatorname{Sch}_\C}(Y \times_\C \operatorname{Spec}(\mathbb \C[t] /(t^k)) , X)$$
and the $k$-jet scheme is defined as the scheme representing the corresponding functor $ X_k : \operatorname{Sch}_\C^{op} \to \operatorname{Set} $ defined on objects by $ Y \mapsto X_k(Y)$.
\end{df}

This definition relies on the representability result of the functor $X_k$, which one can find in \cite[Proposition 2.1.3]{greenbook}. The scheme representing the functor $X_k$ will also be denoted by $X_k$ and it will be called the $k$-jet scheme, and we call $k$-jet a point of $X_k$. The $\C$-points of $X_k$ are given by $X_k(\C) = \operatorname{Hom}_{\operatorname{Sch}_\C}(\operatorname{Spec}(\C[t] /(t^k)) , X)$ and correspond to $\C[t] / (t^k)$-points of $X$.

\begin{df}
Let $l \geq r >0$ be two integers and consider the map $\pi^l_r : \mathbb K[t] / (t^l) \to \mathbb K[t] / (t^r)$ of reduction modulo $t^r$. We denote also by $\pi^l_r: X_l \to X_r$ the induced morphism of schemes.
\end{df}

Following \cite[Section 3.3.1, Corollary 3.3.7]{greenbook} the truncation morphisms $(\pi^{k+1}_k)_{k \geq 0}$ form a projective system which admits a limit in the category of functors over $\operatorname{Sch}_\C$. The limit functor is again representable and the scheme representing it can be taken as the definition of the arc scheme. In particular, we summarize the key points of the arc scheme construction in the next definition.

\begin{df} \label{def-the arc scheme}
The arc scheme is defined as the scheme representing the limit functor $\varprojlim X_k = X_\infty : \operatorname{Sch}_\C^{op} \to \operatorname{Set}$. In particular, the functor on object is defined as $$X_\infty(Y) = \operatorname{Hom}_{\operatorname{Sch}_\C}(Y \times_\C \operatorname{Spec}(\mathbb \C[[t]]) , X).$$
\end{df}

The scheme representing the functor $X_\infty$ will also be denoted by $X_\infty$, and we call arc a point of $X_\infty$. 
We observe that in this limit case as well the $\C$-points of $X_k$ are given by $X_\infty(\C) = \operatorname{Hom}_{\operatorname{Sch}_\C}(\operatorname{Spec}(\C[[t]] ) , X)$ and correspond to $\C[[t]]$-points of $X$ as explained in \cite[Remark 3.3.9]{greenbook}.

\begin{df} \label{tr-arcs}
Let $ r >0$ be an integer and consider the map $\pi_r : \mathbb K[[t]] \to \mathbb K[t] / (t^r)$ of reduction modulo $t^r$. We denote also by $\pi_r: X_\infty \to X_r$ the induced morphism of schemes.
\end{df}

\begin{ex} \label{the jet scheme-aff}
We present the example of the affine space and more in general affine varieties since this ultimately is our local model of varieties.
Let $X = \mathbb A^n$ with coordinates $x_1, \dots , x_n$. We have that a morphism $ \phi : \C[x_1, \dots , x_n] \rightarrow \C[t] / (t^k)$ corresponds to its images $\phi(x_i)=a^i \in \C[t] / (t^k)$, each given by a $k$-tuple of coefficients $(a^i_0, \dots, a^i_{k-1}) \in \mathbb \C^k$.
Therefore, we have that 
$$ (\mathbb A^n)_k = (\mathbb A^k)^n = \mathbb A^{nk}.$$
Moreover, considering $Y \subseteq X$ an affine subvariety determined by $I=(f_1, \dots , f_r)$ its $k$-jet scheme is then the affine subscheme of $(\mathbb A^n)_k(\C)$ determined by the conditions $f_j(a^i) = 0$ modulo $t^k$ for $j= 1, \dots , r$.
\end{ex}

Next, we turn our attention to the notion of motivic classes of varieties. These classes are fundamental to defining the motivic measure introduced by Denef and Loeser in \cite{mea} for varieties over $\C$ since it takes value into the ring of motivic classes of varieties.

\begin{df}
The Grothendieck ring of varieties is defined as the quotient of the free abelian group generated on the set of isomorphism classes of varieties over $\C$ modulo the subgroup generated by the relationship
    \begin{center}
    $[X] = [Z] + [X-Z]$ for any $X \in \operatorname{Var}_\C$ and $V \subseteq W$ Zariski closed  
    \end{center}
and then equipped by the multiplication defined as
    \begin{center}
    $[X] [Y] = [X \times Y]$ for any $X,Y \in \operatorname{Var}_\C$.    
    \end{center}
The elements of this ring are called motivic classes and we denote the ring by $\mathbf{K_0}(\operatorname{Var}_\C)$ and by $\Li = [\mathbb{A}^1]$ the motivic class of the affine line.
\end{df}

We recall the following well-known and useful facts, from \cite[Proposition 5]{mot} and \cite[Lemma 2.9]{brid}.
\begin{prp} \label{fib}
Let $X,Y $ be varieties and let $f: X \to Y$ be a fibration of fiber $F$ which is locally trivial for the Zariski
topology of $Y$. Then, $[X] = [F][Y]$.
\end{prp}

\begin{prp} \label{geom}
Let $X,Y$ be varieties and let $f: X \to Y$ be a geometric bijection, that is a morphism which is also a bijection $f(\C) : X(\C) \to Y(\C)$ of $\C$-points. Then, $[X] = [Y]$.
\end{prp}

The construction of the motivic measure of constructible subsets of the arc scheme is based on the result \cite[Theorem 7.1]{mea} by Denef and Loeser, where the authors prove that the limit considered in the next definition exists.
As required by their result, we need to work with the localization $\mathbf{M_0} = \mathbf{K_0}(\operatorname{Var}_\C)[\Li^{-1}]$ as we need $\Li$ to be invertible and with the completion $\widehat{\mathbf{M}}_{\mathbf{0}}$ of $\mathbf{M_0}$ with respect to the filtration $F^k\mathbf{M_0}$, where $F^k\mathbf{M_0}$ denotes the subgroup generated by classes of varieties $\displaystyle\frac{[X]}{\mathbb{L}^i}$ with $\dimm X -i \leq -k$. 

\begin{df} \label{def-measure}
Let $X$ be a variety of pure dimension $n$.
A constructible subset in $X_k$ is a union of locally Zariski closed sets in $X_k$ and a constructible subset in the arc scheme $X_\infty$ is a subset of the form $A=\pi_k^{-1}(C_k)$ for $C_k \subset X_k$ constructible and some $k > 0$.
Then, the limit $$ \mu (A) = \lim_{ m\to \infty} \frac{[\pi_m(A)]}{\Li^{\dimm X m}}  = \lim_{m \to \infty} \frac{[\pi_m(A)]}{\Li^{n m}} $$ exists in $\widehat{\mathbf{M}}_{\mathbf{0}}$ and we call it the motivic measure of $A$.
\end{df}

\begin{rmk}
As observed in \cite[Definition-Proposition 3.2]{mea}, in the case of a smooth variety $X$ for $A=\pi_k^{-1}(C_k)$ a constructible set in the arc scheme $X_\infty$ the definition of motivic measure simplifies to $$\mu (A) = \displaystyle\frac{[\pi_k(A)]}{\Li^{n k}}.$$
\end{rmk}

\begin{rmk}
As also discussed in the accompanying result \cite[Definition-Proposition 3.2]{mea} the definition of $\mu$ extends from constructible subsets to a $\sigma$-additive measure $\mu$ on the Boolean algebra of constructible subsets of the arc scheme.  
\end{rmk}

We can now define integrals with respect to this measure as certain converging series in $\widehat{\mathbf{M}}_{\mathbf{0}}$.
\begin{df} \label{def-int}
Let $X$ be a variety and let $A$ be a set in the arc scheme $X_\infty$. Let $F : A \rightarrow \mathbb{N}$ be a function with constructible subsets as fibers, that we call an integrable function. Its motivic integral is then defined as the series $$\displaystyle\int_{A} \Li^{-F} = \displaystyle\sum_{m=0}^\infty \mu(F^{-1}(m)) \Li^{-m} \in \widehat{\mathbf{M}}_{\mathbf{0}}.$$ 
\end{df}

It bears noting that motivic integrals admit the same change of variables formula as ``regular'' integrals.
Before presenting the change of variable formula we need to consider a specific example of an integrable function, the order function, about which we refer the reader to \cite[Chapter 4, Section 4.4.4]{greenbook} for any further detail.

\begin{df} \label{ord}
Let $W$ be a variety and let $V \subseteq W$ be a closed subvariety. 
We consider the integrable function $ord_V : W_\infty \setminus V_\infty \to \mathbb N$ returning the order of arcs along the non-zero coherent sheaf of ideals defining $V$. Concretely, if $\mathcal{I}$ denotes the latter and $p_\gamma = \gamma (\operatorname{Spec}(\C)) \in V$ then $$ord_V (\gamma) = \operatorname{inf} \{ val_t(\gamma^*(f)) \, | \, f \in \mathcal{I}_{p_\gamma} - \{ \mathbf{0} \}  \} $$ where $val_t$ denotes the $t$-adic valuation (or simply valuation). 
We often refer to $ord_{p_\gamma}(\gamma)$ as the valuation of the arc.
\end{df}

\begin{rmk}
From now on we will omit removing $V_\infty$ and just write $W_\infty$ since $\mu(V_\infty) =0$, see \cite[Chapter 7, Section 3, Subsection 3.3]{greenbook}.   
\end{rmk}

\begin{ex}
We can make the order function along subschemes more explicit since we will locally work with affine varieties, as presented in this example. For $V$ an affine variety given by the ideal $I(V) = (f_1 , \dots , f_r) \subseteq \C[x_1, \dots , x_n]$ the order function is explicitly defined as $ord_V : W_\infty \to \mathbb N , \, \gamma \mapsto ord_V (\gamma) = \operatorname{min}_i \{ val_t(f_i(\gamma)) \}$.
\end{ex}

In particular, we consider the order function of the coherent sheaf of ideals corresponding to the relative canonical divisor of a birational transformation and use it as the correcting factor to change variables in motivic integrals.
\begin{thm}[{\cite[Theorem 4.3.1]{greenbook}}] \label{thm-change}
Let $X,Y$ be smooth varieties. Let $h : Y \to X$ be a proper birational morphism and let $h_\infty : X_\infty \to Y_\infty$ be the induced morphism on the arc scheme. Let $K_{ Y | X} = K_Y - h^*(K_X)$. Then, for $A \subseteq X_\infty$ a set and $F :A \to \mathbb N$ an integrable function we have
$$\displaystyle\int_{A} \Li^{-F} d\mu = \displaystyle\int_{h^{-1}(A)} \Li^{-F \circ h_\infty |_{h^{-1}(A)} - ord_{K_{ Y | X}}} d\mu .$$
\end{thm}

The change of variable formula just presented will be used in Section \ref{denef} together with the theory of embedded resolution of singularities to directly relate a certain Igusa zeta function to the geometric data of the resolution of the singularity, as well as giving the integral an alternative explicit description since the measure of sets $ord^{-1}_V(m)$ is usually challenging to compute. An explicit formulation obtained applying Theorem \ref{thm-change} to an embedded resolution is what we call a Denef's formula, after \cite{denfor}.

To conclude this section, we use the order function to also introduce a specific class of parametric motivic integrals called Igusa zeta functions, taking values in the ring $\widehat{\mathbf{M}}_{\mathbf{0}}[[T]]$ of formal power series with coefficients in $\widehat{\mathbf{M}}_{\mathbf{0}}$. The formal variable $T$ is usually written as $\Li^{-s}$ for $s$ a formal parameter.
\begin{df} \label{def-igusa}
Let $W$ be a variety and let $V \subseteq W$ be a closed subvariety. Let $A \subseteq W_\infty$ be a constructible subset.
The motivic Igusa zeta function of $V$ over $A$ is defined as $$\displaystyle\int_{A} T^{ord_V} d\mu = \sum_{m = 0}^\infty\mu(ord_V^{-1}(m))T^m \in \widehat{\mathbf{M}}_{\mathbf{0}}[[T]] .$$ 
\end{df}

\subsection{Embedded resolution of singularities} \label{resol}
This study of singularities was developed by Hironaka, converging in his result regarding embedded resolutions. We refer the reader to \cite{resol} for any further detail.

\begin{df}
Let $D$ be a divisor of a fixed variety, that is a closed subvariety of pure codimension $1$, and let $D_i$ for $\in i \in I$ the set of its irreducible components. We say that $D$ has simple normal crossing if the irreducible components of $D$ are smooth and meet transversally.
For $J \subseteq I$ we denote with $D_J =  \displaystyle\bigcap_{j \in J} E_j$ and with $D_J^o = D_J  \setminus \displaystyle\bigcup_{j \not\in J} E_j$.
\end{df}

\begin{df} \label{res}
Let $W$ be a variety and let $V \subseteq W$ be a closed subscheme.
Then, an embedded resolution of the pair $(W,V)$ is a pair $(Y,\pi)$ where $\pi: Y \rightarrow X$ is a proper birational morphism, $Y$ is smooth, and $\pi^{-1}(V)$ is a simple normal crossing divisor on $Y$.
\end{df}
We call strict transform $\widetilde{V}$ of $V$ the component of $\pi^{-1}(V)$ that corresponds to the closure of the dense subset of $V$ where $\pi$ is an isomorphism. We usually denote the rest of the components by $E = \displaystyle\bigcup_{i \in I } E_i$ and we call them exceptional divisors, and we do not consider $\widetilde{V}$ as part of this set.

\begin{thm}[{\cite[Main Theorem II(N)]{resol}}]
Let $\C$ be a field of characteristic zero and $(W,V)$ a variety and closed subvariety respectively over $\C$. Then, there exists an embedded resolution $(Y, \pi)$ of the pair such that it is an isomorphism outside of $W_s \cup V$, where $W_s$ denotes the singular locus of $W$. Moreover, the morphism $\pi$ is a composition of blow-ups with smooth centers. 
\end{thm}

By utilizing the transversality condition, we can deduce certain numerical information of the singularity as presented in \cite[Chapter 7, Theorem 3.3.4]{greenbook}. Before that, we recall the notion of multiplicity of a non-zero coherent sheaf of ideals at a point.
\begin{df}
We recall that given a variety $X$ and a point $p \in X$, the multiplicity of a non-zero coherent sheaf of ideals $\mathcal{I}$ on $X$ at $p$ is defined as $$ \operatorname{max} \{ \mu \, | \, \mathcal{I}_p \subseteq m_p^\mu \} $$ where $m_p \subseteq \mathcal{O}_{X, p}$ denotes the maximal ideal of $p$.
\end{df}

\begin{df} \label{ddef-data} \label{mult}
Let $\pi: (Y, E) \rightarrow (W,V)$ be an embedded resolution of the pair $(W,V)$ as in Definition \ref{res} and let $\mathcal{I}$ be the non-zero coherent sheaf of ideals determining $V$.
For every $i \in I$ we denote by $N_i$ the multiplicity of $\mathcal{I} \, \mathcal{O}_Y$ along any smooth point of $E_i$ and by $\nu_i$ the multiplicity of $K_{Y | W}$ along any smooth point of $E_i$. These are respectively called multiplicity and discrepancy of $E_i$.
\end{df}

We conclude the section with a classic example, whose details can be found in \cite[Example 3.1.9]{greenbook}.
\begin{ex} \label{ex-cusp}
We briefly recall the resolution of singularities of the cusp $f = x_1^2 - x_2^3 \in \C[x_1, x_2]$. One can check that three blowups are needed, thus obtaining three exceptional divisors in addition to the strict transform of the curve for which the multiplicities and discrepancies are $(N_i, \nu_i)_{i=1,2,3} = (2, 1),(3, 2),(6, 4)$.
\end{ex}

\subsection{The Hilbert scheme of points of a variety}

The Hilbert scheme of points is defined as the moduli space parameterizing zero-dimensional subschemes of a given one with fixed colength, originally introduced in \cite{EGA}. We recall its definition starting with the construction as a functor of points and then present its representability in the category $\operatorname{Sch}_\C$ of schemes.

\begin{df} \label{def-fun}
Let $X$ be a scheme. For any $n \geq 1$ we consider the functor $H_X^n : \operatorname{Sch}_\C^{op} \to \operatorname{Set}$ 
defined on objects of $\operatorname{Sch}_\C$ as 
$$ S \mapsto H_X^n(S) = \{ Z \subseteq X \times S \, | \,  p|_Z : X \times S \to S \text{ finite, flat and surjective of degree } n \}$$ and on morphisms given $f: T \to S$ morphism in $\operatorname{Sch}_\C$ as 
\begin{center}
    $H_X^n(f) : H_X^n(S) \to H_X^n(T), \,  Z \mapsto (id_X , f)^{-1} (Z) = \,$ pullback with respect to $(id_X , f) : X \times T \to X \times S$.
\end{center}
\vspace{2mm}       
\begin{center}
\begin{tikzcd}
Z\times_S T \subseteq X \times T \arrow[d, "p_1"'] \arrow[r, "p_2"] & T \arrow[d, "f"] \\
Z \subseteq X \times S \arrow[r, "p|_Z"] & S    
\end{tikzcd}
\begin{tikzcd}
{(x,t)} \arrow[d, maps to] \arrow[r, maps to] & t \arrow[d, maps to]\\
{z=(x,s)} \arrow[r, maps to] & s=f(t)
\end{tikzcd}
\end{center}
\end{df}

The representability of the so-called Hilbert functor $H_X^n$ has been proved in \cite[Theorem 3.2]{EGA}. This is a key fact, as our primary focus will rely on the scheme representing it.

\begin{df} \label{def-hilb}
We denote by $\Hilb^n(X)$ the scheme representing the functor $H_X^n$ and we call it the Hilbert scheme of points of $X$ of colength $n$.
The universal subscheme is the subscheme $\mathcal{Z} \in H_X^n(\Hilb^n(X))$ corresponding to $id_{\Hilb^n(X)}$ under representability. 
\end{df}

\begin{rmk}
Using the representability we have a correspondence $$pt \mapsto \mathcal{Z} \cap ( X \times \{ pt \})$$ between closed points of $\Hilb^n(X)$ and zero-dimensional subschemes $Z \subseteq X$ with finite length equal to $\dimm \mathcal{O}_Z = n$.
One can view these subschemes as points on the Hilbert scheme.
\end{rmk}

\begin{df}
    Given a zero-dimensional finite-length subscheme $Z \in \Hilb^n(X)$ we call its support the set of points $|Z| = \{ x_1 , \dots , x_k\}$ where the scheme is geometrically supported, and we define the cluster associated to $Z$ as the $0$-cycle $[Z] = \displaystyle\sum_{x_i \in |Z|} \dimm \mathcal{O}_{Z, x_i} \, x_i$.
\end{df}

Constraining the support, we can define punctual schemes and {the punctual Hilbert scheme of points} as the ones with support concentrated in one point of higher multiplicity. 

\begin{df} \label{d-h}
Let $X$ be a scheme and let $p \in X$. We denote by $$\Hilb_p^n(X) = \{ Z \in \Hilb^n(X) \, | \, |Z| = \{ p \} \}$$ the subscheme of $\Hilb^n(X)$ made of the zero-dimensional subschemes of length $n$ that are supported entirely at the point $p$.
\end{df}

\begin{rmk}
From now on, we will omit the word ``punctual'' and we will always refer to the Hilbert scheme as the one of Definition \ref{d-h}.  
\end{rmk}

The Hilbert scheme still has a scheme structure since it can be equivalently defined as a fiber of the Hilbert-Chow morphism $Z \mapsto |Z|$. We refer to \cite{HC} for the construction of this morphism.

Returning to our local case of interest which is affine varieties and hypersurfaces, the Hilbert scheme can be described in an even clearer way.

\begin{ex} \label{ex-ide}
Given an affine variety $(X,0)$ the Hilbert scheme takes the form 
 $$\Hilb^n_0(X)=\{ I \subseteq  \mathcal{O}_{X,0} \text{ ideal} \, | \dimm \mathcal{O}_{X,0} / I = n\}.$$  
Moreover, for $X = \{ f=0 \}$ an affine hypersurface given by $f \in \C[x_1 , \dots , x_n]$ we have $$\Hilb^n_0(X) = \{ J \subseteq \C[x_1 , \dots , x_n] \text{ ideal} \, | \, f \in J \, , \, J \subseteq m \, , \, \dimm \C[x_1 , \dots , x_n] / J = n \}.$$
\end{ex}

\section{Curvilinear Hilbert schemes and motivic integration} \label{curv}

We now move onto a specific type of points of the Hilbert scheme, the {curvilinear} ones, mentioned already in Section \ref{intro}.
The definition of curvilinear schemes appears for example in \cite{Berczi}, which we refer the reader to.

\begin{df} \label{def-chilb}
Let $(X,0)$ be a variety and $Z \in \Hilb_0^k(X)$. We say that $Z$ is curvilinear if it lies on a smooth curve $C_Z \subseteq X$, or equivalently if $$\mathcal{O}_{Z} \cong \C[t] / (t^k)$$ are isomorphic as $\C$-algebras.
The open consisting of curvilinear schemes is called curvilinear Hilbert scheme and is denoted by $\CHilb_0^k(X)$.
\end{df}

As shown in \cite[Section 3]{Berczi}, the curvilinear part $\CHilb_0^k(X)$ is an open subscheme of the Hilbert scheme of points $\Hilb_0^k(X)$.
We now prove a well-known property that characterizes the curvilinear schemes of $\mathbb A^n$, which will also be useful to control the generators of the ideals of Example \ref{ex-ide}. We work with completions as explained in Remark \ref{rmk-compl}. 

\begin{lemma} \label{prop-chilb}
Let $I  = (g_1, \dots , g_r)  \in \CHilb^k_0(\mathbb A^n)$. Then, $I \not\subseteq m^2$ where $m = (x_1, \dots , x_n) \subseteq \C[[x_1, \dots , x_n]]$ denotes the maximal ideal in the ambient coordinate ring. In particular, for degree reasons, there exist $n-1$ linearly independent generators $g_i \in m \setminus m^2$ of $I$ and $g_j \in m^n$ for any remaining $j$.
\end{lemma}

\begin{proof}
By definition, there exists a smooth curve $C$ where the subscheme corresponding to $I$ lies. Let $I_C = (f_1, \dots , f_r)$ be the ideal defining $C$. Again by definition, we have $I_C \subseteq I$, and we then know by the jacobian criterion that $f_i \in m \setminus m^2$ for some $n-1$ distinct $i$. So, $f_i \in I$ implies $I \not\subseteq m^2$.

Then, we now know that $n-1$ generators of $I$ are linear, say $g_1 , \dots , g_{n-1}$ for simplicity. Consider any generator $g_j$ with $j \neq 1, \dots n-1$. Then, again for the jacobian criterion we have that $$\C[[x_1, \dots , x_n]] / (g_1 , \dots , g_{n-1}) \cong \C[[t]]$$ for a certain isomorphism $\phi$ and it follows that $n \leq \dimm \C[[x_, \dots , x_n]] / (g_1 , \dots , g_{n-1}, g_j )$ and $\phi(m)=(t)$.
If $g_j \in m^{n-1} \setminus m^n$ then $\phi(g_j) \in (t^{n-1}) \setminus (t^n)$ and $$\dimm \C[[x_1, \dots , x_n]] / (g_1 , \dots , g_{n-1}, g_j ) \leq n-1$$ which is a contradiction.
\end{proof}

\subsection{The non-reductive quotient structure of the jet scheme and their motivic classes} \label{the jet scheme} \label{31}

To understand the relationship between motivic integrals and Hilbert schemes, we need to start with jet schemes since these are the ``building blocks'' of motivic integrals as observed in Section \ref{back}. In particular, we need to construct and investigate in detail the motivic class of a certain geometric quotient of the jet schemes of the given variety.

\begin{df}
Let $(X,0)$ be a variety. We consider the set in $X_k$ of smooth, punctual $k$-jets $$S_0^k(X)=\{ \phi : \operatorname{Spec} (\C[t] / (t^k)) \hookrightarrow X \, \text{closed embedding} \, | \, \phi ( \operatorname{Spec}(\C)) = 0 \} .$$
\end{df}

By utilizing the algebraic characterization presented in Definition \ref{def-chilb} in terms of isomorphic $\C$-algebras, we can swiftly recognize how significant it is to focus on the smooth curve where curvilinear schemes lie.
We specifically use the jet scheme to describe curves and represent curvilinear schemes, and it is enough to consider them up to smooth reparameterization since it does not affect those curves as geometric loci.

\begin{df}  \label{def-quot} \label{action}
Let $(X,0)$ be a variety.
Noticing that any $\alpha \in (t) \setminus (t^2) \subseteq \C[t]/(t^k)$ induces the ring automorphism $\alpha$ of $\C[t] / (t^k)$ defined by $$t \mapsto \alpha$$ we denote by $R^k_{s,0}$ the group of such automorphisms.
 On $S_0^k(X)$ we consider the action induced by $R^k_{s,0}$ and the composition defined as 
    \begin{center}
        $\bullet : R^k_{s,0} \times S_0^k(X) \to S_0^k(X)$, $(\alpha, \phi) \mapsto \phi \bullet \alpha = \phi \circ \operatorname{Spec}(\alpha)$. 
    \end{center}    
\end{df}

Since throughout the rest of this paper we will work locally around a fixed point of a smooth ambient variety, eventually working in local coordinates is enough.

\begin{rmk} \label{rmk-compl}
To choose local coordinates around a point of a smooth variety (or more in general, smooth at the considered point) we can proceed as follows.
Given $(X,0)$ a smooth variety of dimension $n$, 
we can consider its completed local coordinated ring and being smooth equivalently means that $$\widehat{\mathcal{O}}_{X,0} \cong \C [[x_1, \dots , x_n]].$$ 
Observing that the Hilbert scheme is invariant under the choice of working with the local coordinate ring or its completed version, we can consider the latter and the regular functions $x_1, \dots , x_n$ as coordinates around $0$ on $X$. Indeed, any morphism $Z \to \operatorname{Spec}(\mathcal{O}_X) \times S$ factors through $\operatorname{Spec}(\widehat{\mathcal{O}}_{X,0}) \times S$ and $Z \to \operatorname{Spec}(\widehat{\mathcal{O}}_{X,0}) \times S$ is still a closed immersion.
\end{rmk}

\begin{rmk} \label{loc}
Assuming $X$ smooth and choosing local coordinates $x_1, \dots , x_n$ for $X$ around $0$ the elements of $S_0^k(X)$ are described equivalently as morphisms $$ \phi: \mathcal{O}_{X,0} \twoheadrightarrow \C[t] / (t^k)$$ corresponding to $n$-tuples $(\phi_1 = \phi(x_1) , \dots , \phi_n = \phi(x_n))$ of truncated power series in $\C[t] / (t^k)$ with zero degree term equal to $0$. This means that $$\phi_i = a_1^i t + \dots + a_{k-1}^i t^{k-1}$$ for $i=1, \dots, n$.
For this reason, we can think of these the jet scheme as ``starting'' at $0 \in X$. 
In addition to this, we observe that the surjectivity condition corresponds to the constraint $$(\phi_1 , \dots , \phi_n) \in (t)^n \setminus (t^2)^n \subseteq (\C[t] / (t^k))^n$$ which translates to the smoothness of the geometric locus parameterized by $\phi$. This means that $$a_1^i \neq 0$$ for some $i=1, \dots, n$.
For this reason, we will equivalently call smooth a surjective jet.

Moreover, for $\alpha \in R^k_{s,0}$ the action $\bullet$ translates to the composition of $\phi_i$ with $\alpha$ for $i=1, \dots, n$.
\end{rmk}

Moreover, we already know from Example \ref{the jet scheme-aff} that elements of the jet scheme correspond to their coefficients and the action takes the following multiplication-by-matrix form, as presented in \cite[Lemma 3.4 and Remark 3.5]{Berczi} or \cite[Section 5]{berkir} alternatively for $J_k(1,n)(X)$ and that we restate for $X_k$.

\begin{lemma} \label{prop-action}
    Let $k \geq 2$ and let $x_1, \dots , x_n$ be local coordinates for $X$ around $0$ a smooth variety of dimension $n$.
    Let $\phi \in S_0^k(X)$ and $\alpha \in R^k_{s,0}$ an automorphism of $\C [t] / (t^k)$.   
    Then, identifying $\phi$ with the matrix of its coefficients $(a_j^i) \in (\C^k)^n$, the action $\bullet$ of Definition \ref{action} translates into the matrix multiplication
\vspace{1mm}
$$ 
(p_1, \dots , p_n) \bullet \alpha = 
\begin{pmatrix}
a_1^1 & a_2^1 & \dots & a_{k-1}^1\\
a_1^2 & a_2^2 & \dots & a_{k-1}^2\\
\dots & & & \dots\\
a_{k-1}^n & a_{k-1}^n & \dots & a_{k-1}^n
\end{pmatrix}
\bullet
\begin{pmatrix}
\alpha_1 & \alpha_2 & \dots & \alpha_{k-1}
\end{pmatrix}
=$$
\vspace{2mm}
$$
\begin{pmatrix}
a_1^1 & a_2^1 & \dots & a_{k-1}^1\\
a_1^2 & a_2^2 & \dots & a_{k-1}^2\\
\dots & & & \dots\\
a_{1}^n & a_{2}^n & \dots & a_{k-1}^n
\end{pmatrix}
\cdot
\begin{pmatrix}
\alpha_1 & \alpha_2 & \alpha_3 & \dots & \alpha_{k-1}\\
0 & \alpha_1^2 & 2 \alpha_1 \alpha_2 & \dots & \dots\\
\dots & &  & & \dots\\
0 & 0 & \dots & \dots & \alpha_1^{k-1}
\end{pmatrix}
$$
and we denote by $\alpha^{i , j} = \displaystyle\sum_{a_1 + \dots + a_i = j} \alpha_{a_1} \dots \alpha_{a_i}$ the $(i,j)$-entry of the matrix corresponding to $\alpha$.
\end{lemma}

\begin{rmk}
As pointed out in \cite{Berczi} the subgroup of these matrices corresponds to a certain $\mathbb G_m$-extension of the maximal unipotent radical $U$ of $R^k_{s,0}$ that is obtained imposing $\alpha_1 =1$. Therefore, considering the semi-direct group structure of $R^k_{s,0}$ and the underlying product space, this space is isomorphic to $\mathbb G_m \times \mathbb A^{k-2}$ as a variety thus $$[R^k_{s,0}] = [ \mathbb G_m \times \mathbb A^{k-2}] = (\Li - 1)\Li^{k-2}.$$ 
\end{rmk}

Some attention is needed to perform the quotient of $S_0^k(X)$ by the action of Definition \ref{def-quot} and Lemma \ref{prop-action} since the action is non-reductive. In particular, we need to ensure we obtain a geometric quotient in the sense of \cite{mum}.
The precise theory and construction of geometric quotients of the jet scheme by the non-reductive group of smooth reparameterizations is developed in \cite{berkir} and we avoid recalling any detail here. 
\begin{df}
We denote by $Q_0^k(X)$ the geometric quotient of $S_0^k(X)$ by the action $\bullet$, that is by $R^k_{s,0}$.
\end{df}

However, even the notion of geometric quotients is generally not enough to obtain information about the motivic classes of quotient spaces. We are especially interested in equipping the quotient map with the structures of principal bundle and Zariski locally trivial fibration as we aim to use Proposition \ref{fib} to prove relationships between motivic classes of the jet scheme $S^k_0(X)$ and its quotient by the action $\bullet$.
We start by noticing that from the description of Lemma \ref{prop-action} the action is scheme-theoretically free. 

\begin{prp}
The action described in Lemma \ref{prop-action} is scheme-theoretically free.
\end{prp}

\begin{proof}
First of all, the action explicitly takes the form
$$ 
(p_1 , \dots , p_n) \cdot \alpha = 
\begin{pmatrix}
a_1^1 \alpha_1 & a_1^1 \alpha_2 + a_2^1 \alpha_1^2 & \dots & \sum_i a^1_i \alpha^{i , k-1}\\
a_1^2 \alpha_1 & a_1^2  \alpha_2 + a_2^2 \alpha_1^2  & \dots & \dots\\
\dots & & & \dots\\
a_{1}^n \alpha^{i , k-1} & a_{1}^n  \alpha_2 + a_2^n \alpha_1^2  & \dots & \sum_i a_i^n \alpha^{i , k-1}
\end{pmatrix}
$$
and these polynomial equations determine the image of the natural map $$\Psi : S_0^k(\mathbb A^n)  \times R_{s,0}^k \to S_0^k(\mathbb A^n)  \times S_0^k(\mathbb A^n).$$ Then, we can observe that for $a_1^j \neq 0$ for some $j$ the equation $a_1^j \alpha_1 = a_1^j$ implies $\alpha_1 =0$ and the equations $a_1^j \alpha_2 + a_2^j \alpha_1^2 = a_2^j ,  \, \dots \, , \displaystyle\sum_i a_i^j \alpha^{i , k-1} = a_n^j$ in chain imply $\alpha_2=0, \, \dots \, , \alpha_{k-1}=0$. This means the action is set-theoretically free.
Moreover, the map $\Psi$ is injective 
and so to prove that it is a closed immersion we have to prove it admits a local inverse on its image. In particular, following the notation of Definition \ref{prop-action} on the open set $U_j$ determined by the condition $a^j_1 \neq 0$ we construct the inverse of $\Psi$. 
Given a pair $(p,q) \in \operatorname{Im}(\Psi)$ and denoting by $a_i^j, b_i^j$ the coefficients of $p, q$ respectively, to lie in $\operatorname{Im}(\Psi)$ there needs to be an $\alpha \in R_{s,0}^k$ such that $p \cdot \alpha = q$ and to construct the local inverse we need to prove that there exists a unique, algebraic way to determine $\alpha$.
From the equation $a_1^j \alpha_1 = b_1^j$ we obtain $\alpha_1 = \frac{b_1^j}{a_1^j}$, then from the equation $a_1^j \alpha_2 + a_2^j \alpha_1^2 = b_2^j$ we obtain $\alpha_2 = \frac{1}{a_1^j} \left( b_2^j - a_2^j \left( \frac{b_1^j}{a_1^j} \right)^2 \right)$ and so on until we determine $\alpha_{k-1}$ in terms of $p_j , q_j$.
Therefore, the local inverse of $\Psi$ is given by $$(p,q) \mapsto \left(\frac{b_1^j}{a_1^j} \, , \, \frac{1}{a_1^j} \left( b_2^j - a_2^j \left( \frac{b_1^j}{a_1^j} \right)^2 \right) \, , \, \dots \right)$$ and it is clearly algebraic.
\end{proof}

From this proposition, we can apply \cite[Proposition 0.9]{mum} and obtain the additional structure of a principal $R^k_{s,0}$-bundle for the geometric quotient $Q_0^k(X)$. 

\begin{crl}
    In the setting of Definition \ref{def-quot}, the quotient map
    $p_k: S_0^k(X) \to Q_0^k(X)$ is a principal $R^k_{s,0}$-bundle.
\end{crl}

We now move on to studying the fibration structure. One of the classical tools to make principal bundles into Zariski locally trivial fibrations is exploiting the theory of special groups, developed in \cite{serre} to which we refer the reader for any detail.

Observing that both $\mathbb G_m$ and $U$ of Lemma \ref{prop-action} are special groups, so is their extension $R^k_{s,0}$ as developed in \cite[Section 4]{serre}. Lastly, being a special group the next result immediately follows.

\begin{crl}
    The quotient map
    $p_k: S_0^k(X) \to Q_0^k(X)$ is a Zariski locally trivial fibration.
\end{crl}

Finally, from Proposition \ref{fib} applied to our Zariski locally trivial fibration $p_k$ we find an explicit description of the motivic class $[Q_0^k(X)]$ in terms of $[S_0^k(X)]$. These classes are both important since the former is strictly related to motivic classes of curvilinear Hilbert schemes and the latter is more suitable to compute certain motivic integrals as we will see in Section \ref{relat}.

\begin{crl} \label{thm-quotient}
Let $(X,0)$ be a smooth variety. For any $k \geq 2$ we have $$[S_0^k(X)] =  [Q_0^k(X)] \cdot [R^k_{s,0}] = [Q_0^k(X)] \cdot (\Li -1) \Li^{k-2}.$$
\end{crl}

\subsection{The relationship between the jet scheme and Hilbert schemes} \label{relat} \label{32}

The goal of this section is to investigate the connection between jet schemes as defined in Definition \ref{def-the jet scheme} and Hilbert schemes as defined in Definition \ref{def-hilb} in the curvilinear case, as there seems to be a natural relationship with truncated power series from their definition.

This comparison starts by constructing a geometric bijection between $Q_0^k(X)$ and $\CHilb^k_0(X)$, to then apply Proposition \ref{geom}. This means we have to construct a morphism that is a bijection at the level of $\C$-points of these varieties.

\begin{thm} \label{thm-iso}
Let $(X,0)$ be a smooth variety over $\C$ and let $$\Phi: Q_0^k(X) \to \CHilb^k_0(X), \, \Phi(\phi) = \operatorname{Im}(\phi) $$ where $\operatorname{Im}$ denotes the scheme-theoretic image. Then, it is a geometric morphism.
\end{thm}

\begin{proof}
Working locally we equivalently work with the description of $Q_0^k(X)$ presented in Remark \ref{loc}, the map $\Phi$ is defined as $\phi \mapsto \operatorname{ker}(\phi)$. We immediately observe that $\Phi$ is a well-defined map since for any given $\phi \in S_0^k(X)$, being surjective onto $\C[t] / (t^k)$ it follows that the quotient by $\operatorname{ker}(\phi)$ is of finite colength and curvilinear.

Then, we have to prove that $\Phi$ is a morphism. 
It is not straightforward to deduce it directly from its definition, so we will use the representability of the more general Hilbert functor of Definition \ref{def-fun}, constructing a morphism that will correspond to the same map on the $\C$-points of $Q_0^k(X)$ and $\CHilb^k_0(X)$. 
We denote by $\mathcal{Z}$ the universal subscheme, corresponding under representability to $id_X$ and consider the subvariety 
\begin{center}
    $Z_{J} \subseteq X \times S_0^k(X)$, $Z_{J }
 = \{ ( supp(\phi) \cap m^k , \phi ) \}$.
\end{center}
 The projection is clearly finite, flat and surjective of degree $n$ - having as fibers the $n$ points of the support of the corresponding parameterized subscheme of colength $n$ at $m$.
We then consider by representability the morphism $\Phi$ such that $Z_{J }= (H(\Phi))(\mathcal{Z}) =  \mathcal{Z} \times_{\CHilb^k_0(X)} S_0^k(X)$.
\begin{center}
\begin{tikzcd}
 & & \mathcal{Z} \arrow[d, hook]\\
Z_J \arrow[r, hook] \arrow[urr, bend left] \arrow[rd] & {X \times S_0^k(X)} \arrow[d, "p"] \arrow[r, "id_X \times \widetilde{\Phi} "] & X \times \CHilb^k_0(X) \arrow[d, "p_2"] \\
                               & {S_0^k(X)} \arrow[r, "\widetilde{\Phi}"]                                     & \CHilb^k_0(X)                          
\end{tikzcd}
\end{center}
 By definition $Z_J$, $\widetilde{\Phi}$ is then defined as associating to a jet the corresponding scheme as a point of the curvilinear Hilbert scheme $\CHilb^k_0(X)$, corresponding to $\Phi$.
 Such morphism factorizes through $R^k_{s,0}$ as it is  $R^k_{s,0}$-invariant, and indeed on $\C$-points it corresponds to the explicit $\Phi$ considered above by construction.

To conclude, we prove that $\Phi$ is a set-theoretical bijection on $\C$-points.
An immediate way to prove it consists in observing that given $I$ curvilinear ideal and $C_I$ the corresponding smooth curve, we have $$\mathcal{O}_{C_I } \cong \C[t].$$ In particular, $I$ is an ideal of $\mathcal{O}_{C_I}$ thus it needs to correspond to another ideal of colength $k$ under the isomorphism: there exists one and only one such ideal up to multiplication by unit, hence one and only one morphism in $S_0^k(X)$ having it as its kernel.
\end{proof}

Finally, applying Proposition \ref{geom} we achieve the equality of motivic classes of curvilinear Hilbert schemes and smooth the jet scheme.
\begin{crl} \label{thm-iso2} 
Let $(X,0)$ be a smooth variety over $\C$. Then,
$$[Q_0^k(X)] = [\CHilb^k_0 (X)].$$  
\end{crl}

It is needed to note that the result of Theorem \ref{thm-iso} thus Corollary \ref{thm-iso2} can also be deduced by the isomorphism presented in \cite[Theorem 3.8, Remark 3.9]{Berczi}, where an algebraic embedding of the quotient $Q_0^k(X)$ into a certain Grassmanian having as image the curvilinear Hilbert scheme $\CHilb^k_0(X)$ is constructed. One just needs to translate such morphism in terms of our notion of jet scheme $$(\mathbb A^n)_k = (\mathbb A^n)_k (\C) = \operatorname{Hom}_{\operatorname{Sch}_\C}(\C[x_1, \dots , x_n], \C[t] / (t^k) )$$  as in Definition \ref{def-the jet scheme} replacing the definition $$ J_k(l,m)(\mathbb A^n) = \operatorname{Hom}_{hol}(\C^l, \C^m) \text{ defined up to its $k$-th derivative}$$ considered in \cite[Subsction 3.1]{Berczi}.
We can immediately observe that given $\phi \in J_k(1,n)(\mathbb A^n) $ by definition it corresponds to a morphism $\phi: \C \rightarrow \C^n$ determined up to its $k$-th derivative, which in turn equals the truncated description $\phi = (p_1(t) , \dots , p_n(t))$ where $ p_i(t)\in \C[t] /(t^k)$.
Hence, $\phi$ can be seen as the morphism of $(\mathbb A^n)_k$ determined by $\phi (x_i)=p_i(t)$. Then, the smoothness and punctual conditions for the jet scheme also translate accordingly.

\subsection{Curvilinear motivic series and curvilinear Igusa zeta function}

Given the relationship between motivic classes of curvilinear Hilbert schemes and of quotient jet schemes proved in Corollary \ref{thm-iso2} and thinking of the definition of motivic integrals as a series of motivic classes of constructible subsets truncated in jet schemes, it is natural for us to introduce the next definition.

\begin{df} \label{def-ser}
Let $(X,0)$ be a smooth variety. We denote by $H_k^X = [\CHilb^k_0(X)]$ the motivic classes of the $k$-th curvilinear Hilbert scheme. We then consider the motivic generating series
$$ Q_X(T) = \sum_{k = 2}^\infty H_k^X \, T^k .$$
\end{df}

\begin{rmk} 
From the definition of the motivic measure and motivic integrals, it would be natural to consider series in motivic classes of (subsets of) jet schemes. In our case, it would be for motivic classes $h_k^X = [S_0^k(X)]$ of smooth, punctual elements of the $k$-jet scheme for every $k \geq 2$. 
We observe that from Theorem \ref{thm-quotient} we obtain the chain of relationships $$ \frac{h_k^X}{(\Li -1)\Li^{k-2}} = H_k^X \iff h_k^X = H_k^X \, \Li^{k-2} ( \Li - 1).$$

What is important to mention is that in Sections \ref{31} and \ref{32} we can also consider more in general an ambient smooth variety $(W,0)$ and a subvariety $(V,0)$ and the relationship 
$$h^V_k = H^V_k \Li^{k-2} (\Li - 1)$$
still holds since $S^k_0(V) = V_k \cap S^k_0(W)$ and $\CHilb^k_0(V)$ corresponds to $Q^k_0(V)$ restricting $\Phi$ to $V_k$.
\end{rmk}

Finally, we have everything needed to present the relationship between Igusa zeta functions and the generating series of motivic classes of curvilinear Hilbert schemes. 

\begin{thm} \label{thm-igusa-Q}
Let $(W,0)$ be a smooth variety of dimension $n$ and let $V \subseteq W$ be a closed subvariety with $0\in V$. Let $A = \pi_1^{-1}(0) \setminus \pi_2^{-1}(0) \subseteq W_\infty$.
Then, we have $$\int_A (\Li^{-ord_V})^s d\mu = 
\frac{1}{\Li^s}\frac{1}{\Li^n}\left(\frac{1}{\Li^n} - \frac{1}{\Li^{2n}}\right)   +   \frac{\Li -1}{\Li^2} (1 - \Li^s) Q_V( \Li^{-s - (n-1)} ).$$ 
\end{thm}

\begin{proof}
We use the notation of Definition \ref{ord}. Consider the following sets in $A$ defined as
$$ A_m = \{ x \in A \, | \, ord_V(x) \leq m \} $$ for $m \geq 1$.
We want to use these sets to build a partition of $A$ that will allow us to compute the integral as a series in motivic classes while being explicitly measurable in the sense of Definition \ref{def-measure}. 
We observe that 
\begin{center}
    $ord_V(x) \leq m $ if and only if $ \pi_m(f_i(x)) =0$ for all $f_i$
\end{center}
that is $f_i(x) = 0$ in $\C[t] / (t^m)$. 
Therefore $$A_m = \{ x \in A \, | \, \pi_m(f_i(x)) = 0 \}  
= \pi_m^{-1} (x \in (\C [t] / (t^m))^n \, | \, f_i(x) = 0 ) \cap A = \pi_m^{-1}(V_m ) \cap A$$ 
thinking of $V_m$ as a subset of $W_m$. Then, when $m > 1$ the key observation is that $\pi_m^{-1}(V_m ) \cap A = \pi_m^{-1}(V_m \cap \pi_m(A))$. 
Thus, denoting by $C_m = V_m \cap \pi_m(A) \subseteq V_m \subseteq W_m$ and observing that it is clearly constructible, we have 
\begin{center}
    $A_m = \pi_m^{-1} (C_m)$ hence $\mu(A_m) = \displaystyle\frac{[C_m]}{\Li^{nm}}.$
\end{center} 
Moreover, we notice that by definition $S^m_0(V) = V_m \cap \pi_m(A)$ implying $[C_m] = h^V_m = H^V_m (\Li - 1) \Li^{m-2}$ and $$\mu(A_m) = \frac{H^V_m (\Li -1)\Li^{m-2}}{\Li^{nm}}.$$ 
When $m = 1$ we have instead $\pi_1^{-1}(Z_m) \cap A = A$.
Thus 
\begin{center}
    $V_1 = \pi_1^{-1} (A)$ hence $\mu(V_1) = \displaystyle\frac{[A]}{\Li^n}= \displaystyle\frac{1}{\Li^n}\left(\frac{1}{\Li^n} - \displaystyle\frac{1}{\Li^{2n}}\right)$.
\end{center}   
Using this to partition the domain of integration and thus the integral, we obtain:
$$\int_A (\Li^{-ord_V})^s d\mu = \frac{1}{\Li^s}\mu(A_1 \setminus A_2) + \frac{1}{\Li^{2s}}\mu(A_2 \setminus A_3) + \frac{1}{\Li^{3s}}\mu(A_3 \setminus A_4) \dots =$$
$$\frac{1}{\Li^s}(\mu(A_1) - \mu(A_2)) + \frac{1}{\Li^{2s}}(\mu(A_2) - \mu(A_3)) + \frac{1}{\Li^{3s}}(\mu(A_3) - \mu(A_4)) \dots =$$
$$\frac{1}{\Li^s}\frac{1}{\Li^n}\left(\frac{1}{\Li^n} - \frac{1}{\Li^{2n}}\right)   +   \sum_{m=2}^\infty \mu(A_m)\left(\frac{1}{\Li^{ms}}  -  \frac{1}{\Li^{(m-1)s}}\right) =$$
$$\frac{1}{\Li^s}\frac{1}{\Li^n}\left(\frac{1}{\Li^n} - \frac{1}{\Li^{2n}}\right)   +   \frac{\Li -1}{\Li^2}\sum_{m=2}^\infty  \frac{H^V_m}{\Li^{m(n-1)}} \left(\frac{1}{\Li^{ms}}  -  \frac{1}{\Li^{(m-1)s}}\right) =$$
$$\frac{1}{\Li^s}\frac{1}{\Li^n}\left(\frac{1}{\Li^n} - \frac{1}{\Li^{2n}}\right)   +   \frac{\Li -1}{\Li^2} (1 - \Li^s) Q_V( \Li^{-s - (n-1)}).$$
\end{proof}

To conclude this section, we check the relationship we just proved between $\displaystyle\int_A (\Li^{-ord_V})^s d\mu$ and $Q_V(T)$ on an easy example.

\begin{ex} \label{ex-line}
In $\mathbb A^2$ we consider the line $(C,0)$ given by $f(x,y)=x$. 
Recurring to usual integration techniques (see \cite{report} for example) we start by computing the left-hand side of Theorem \ref{thm-igusa-Q} and we obtain $$\displaystyle\int_A (\Li^{-ord_C})^s d\mu = \frac{\Li-1}{\Li^2 \Li^{s+1}} \frac{\Li^{s+2}-1}{\Li(\Li^{s+1}-1)}.$$
We then compute the right-hand side. We observe that being smooth $\Hilb^n_0 (C) = \CHilb^n_0 (C)$ and the only possibility for $I$ of colength $n$ is $I=(x,y^n)$. 
The generating series after changing variable as in Theorem \ref{thm-igusa-Q} is then $$Q_C( \Li^{-s - (2-1)}) = Q_C( \Li^{-(s + 1)}) = \displaystyle\sum_{n \geq 2} 1 \left(\frac{1}{\Li^{s+1}} \right)^n = \frac{1}{\Li^{s+1}(\Li^{s+1}-1)}$$
and one can check that the relationship of the theorem holds.
\end{ex}
\begin{ex}
In general, the same computations hold for the class of smooth curves curve defined by $f(x,y)=x-y^k$ (and any other smooth curve) since their local rings are isomorphic to $\C[[y]]$. In particular, all these curves have the same series $Q_C(T)$ with no dependence on $k$ and this is justified in Section \ref{ex}. Indeed, their corresponding link is a $(1,k)$-torus knot, which is always the unknot hence having the same topology.
\end{ex}

\subsection{Denef's formula for the curvilinear Igusa zeta function} \label{denef}

As explained in Section \ref{back}, motivic integrals can be re-expressed in terms of an embedded resolution after changing variables with Theorem \ref{thm-change} and the resolution map, and the result of integration will then depend on motivic classes of divisors appearing from the resolution process since they provide a transverse, locally monomial model of the singularity that is used to compute the integral after the change.

The existing formulations available in the literature are due to Denef \cite[Theorem 3.4]{report} and all concern the so-called residual support functions, that roughly correspond to indicator functions of domains uniquely determined by their reduction modulo $t$. 
This implies that whenever one is interested in the ball $(t^{m+1})$ or the contact locus $(t^m) \setminus (t^{m+1})$ these theorems do not apply.

The curvilinear Igusa zeta function of Theorem \ref{thm-igusa-Q} falls within this latter category, hence forcing us to proceed differently. Specifically, we need to understand what happens to the domain of integration $A = \pi_1^{-1}(0) \setminus \pi_2^{-1}(0)$ under resolution maps, that is how the resolution affects the function $ord_0: W_\infty \to \mathbb N$ on the arc scheme $Y_\infty$ and in particular arcs belonging to $ord_0^{-1}(1)$.
Going back to the data introduced in Definition \ref{ddef-data}, only discrepancies describe the complexity of the resolution independently from the considered singularity but they are not suited to control the behavior of the valuation of arcs along the resolution.
To have that control we can additionally consider the order multiplicity, introduced for example in \cite{Fantini}, that corresponds to the multiplicities of Definition \ref{mult} but for the pair $(W,0)$.


\begin{df} \label{def-order}
Let $\pi: (Y, E) \rightarrow (W,V)$ be an embedded resolution of the pair $(W,V)$ as in Definition \ref{res} that is also a resolution of the pair $(W,0)$ and let $\mathcal{I}$ be the non-zero coherent sheaf of ideals determining $0$.
For every $i \in I$ we denote by $m_i$ the multiplicity of $\mathcal{I} \, \mathcal{O}_Y$ along any smooth point of $E_i$. This is called order of $E_i$.
\end{df}
\begin{df}
We remark here that we do not consider the strict transform as part of the list of divisors, and when wanting to do so one then has $N_{strict} = 1$ and $\nu_{strict} = m_{strict} = 0$. 
\end{df}

\subsubsection{Embedded resolution and the arc-valuation along it} \label{val}

The following observation is needed to start: since the change of variables depends only on the complexity of the resolution map $\pi$ rather than the pair $(W,V)$, it is reasonable to expect a characterization of the domain $\pi^{-1}(A)$ purely in terms of the resolution structure. 

We start by presenting a first result regarding the behavior of the order function $ord_0$ under the morphism of arc schemes $\pi_\infty : W_\infty \to Y_\infty$ induced by an embedded resolution $\pi: (Y, E) \rightarrow (W,V)$ of the pair $(W,V)$ as in Definition \ref{def-order}. 

\begin{prp} \label{thm-val}
Let $\pi: (Y, E) \rightarrow (W,V)$ be an embedded resolution of the pair $(W,V)$ and let $\gamma \in W_\infty$ with $ord_0(\gamma) = k \geq 1$.
Then, $\pi_\infty$ preserves the valuation of arcs, meaning that $$ord_0(\pi_\infty(\gamma)) = k$$
where $\pi_\infty(\gamma) = \pi^*(\gamma) \in Y_\infty$.
\end{prp}
\begin{proof}
We recall that $\pi_\infty: W_\infty \to Y_\infty$ being a bijection is a known fact proved already in \cite[Proposition 4.4.2]{greenbook} and we are left with studying the behavior of $ord_0: W_\infty \to \mathbb N$.
To understand the order in $0$, we recall its definition. Given $\gamma : \operatorname{Spec}(\C [[t]]) \to W$ and given $\mathcal{I}$ the ideal sheaf of $0 \in W$, we have that $\gamma^*(\mathcal{I})$ is an ideal sheaf on $ \operatorname{Spec}(\C [[t]])$. It corresponds to some ideal $$I_\gamma = (t^k)$$ for some $k = ord_0(\gamma)$.
Then, by the bijection $\pi_\infty$ we have a unique lift $\gamma' = \pi^*(\gamma) = \pi_\infty(\gamma)$ as an arc of $Y$ with $\gamma'(\operatorname{Spec}(\C)) = p \in E_J^o$ for some $J \subseteq I$. In particular,  by definition of lift $$\gamma^*(\mathcal{I}) = (\gamma')^* \pi^*(\mathcal{I})$$ and it follows that $k = ord_0(\gamma) = ord_\mathcal{I}(\gamma')$ as the order at $0$ of $\gamma'$ can be retrieved from the stalk $\mathcal{I}_{\gamma'|_{t=0}}$.
\end{proof}

Nevertheless, the previous proposition is not enough as it does not provide any information on how $ord_0(\gamma)$ is specifically related to $ord_{E_j}(\gamma')$ for $j \in J$ where $J$ denotes the subset of divisors where $\gamma'(\operatorname{Spec}(\C)) \in E_J^o$ lies, and this relationship is needed to ultimately understand the order of arcs in terms of the resolution structure. 
From the definition of the order of divisors, we directly have that $(\gamma')^* \pi^*(\mathcal{I})$ corresponds to the ideal $(t^{K})$ where $K = \displaystyle\sum_{j \in J} m_j ord_{E_j}(\gamma')$ thus
\begin{equation}\label{croo}
    ord_0(\pi(\gamma')) = \displaystyle\sum_{j \in J} m_j ord_{E_j}(\gamma').
\end{equation}


Combining Proposition \ref{thm-val} and Equation (\ref{croo}) we deduce a statement that we can think of as a result similar to \cite[Theorem 3.4]{report}, as it allows us to explicitly control certain domains of integration under change of variables. 
Knowing that $\pi_\infty$ preserves the valuation and knowing how the order of divisors contributes to the valuation, we can say which arcs of $Y_\infty$ have valuation $k$ easily in terms of divisors but also correspond to arcs of $W_\infty$ of valuation $k$ under $\pi$.

\begin{crl} \label{crl-den}
Let $\pi: (Y, E) \rightarrow (W,V)$ be an embedded resolution of the pair $(W,V)$ as in Definition \ref{def-order} and let $0 \in V$. Let $(E_i)_{i \in I}$ be the components of the simple normal crossing divisor including the strict transform $\widetilde{V}$. Then, denoting by $A_k$ the set of arcs $\gamma \in W_\infty$ with $ord_0(\gamma) = k$, we have that 
$$\pi^{-1}(A_k) = \displaystyle\coprod_{J \subseteq I} U_J^k$$ 
where
$$U_J^k = \displaystyle\coprod_{J \in I } \{ \gamma' \in Y_\infty \, | \,    
\displaystyle\sum_{j \in J} m_j \, ord_{E_j}(\gamma') = k \text{ and } \gamma'(\operatorname{Spec}(\C)) \in E_J^o\}.$$
\end{crl}

\subsubsection{Curvilinear Denef's formula} To conclude this section we consider $A=A_1$ since it constitutes the domain of integration of the Igusa zeta function studied in Section \ref{curv} and specifically with Theorem \ref{thm-igusa-Q}.
Applying Corollary \ref{crl-den} in combination with the change of variables of Theorem \ref{thm-change}, we immediately notice that only arcs of valuation 1 along divisors of order 1, over $(E_i \setminus \widetilde{V})^o$ or $(E_i \cap \widetilde{V})^o$ for divisors $E_i$ with $m_i=1$, are in bijection with $A$ under $\pi_\infty$. From this, we can get an explicit formulation for the curvilinear Igusa zeta function of Theorem \ref{thm-igusa-Q} in terms of motivic classes of divisors.

\begin{thm} \label{thm-denef}
Let $(W,0)$ be a smooth variety of dimension $n$ and let $V \subseteq W$ be a closed subvariety with $0\in V$.
Let $\pi: (Y, E) \rightarrow (W,V)$ be an embedded resolution of the pair $(W,V)$ as in Definition \ref{def-order} and let $\widetilde{V}$ denote the strict transform of $V$ under $\pi: Y \to W$. 
Then
$$ \displaystyle\int_{A} (\Li^{-ord_V})^s d\mu = \left(\frac{1}{\Li} - \frac{1}{\Li^2} \right)\Li^{-(n-1)} \displaystyle\sum_{i \in S} \frac{1}{\Li^{N_i s + \nu_i}} \left(  [(E_i \setminus \widetilde{V})^o]   +   \frac{1}{\Li^{s+1} - 1} [(E_i \cap \widetilde{V})^o](\Li -1)   \right) .
$$ 
where $S$ denotes the set of divisors of order $1$. 
\end{thm}

\begin{proof}
We recall Theorem \ref{thm-change}, that tells us that changing variables along $\pi$ we have $$\displaystyle\int_{A} (\Li^{-ord_V})^s d\mu = \displaystyle\int_{\pi^{-1}(A)} \Li^{-s \cdot ord_V \circ \pi_\infty |_{\pi^{-1}(A)} - ord_{K_{ Y | X}}} d\mu \; \; \; (\star).$$
We recall that the domain $A = \pi_2^{-1}(0) \setminus \pi_1^{-1}(0)$ corresponds to the subset of arcs of order $1$ in $W_{\infty}$ so from Corollary \ref{crl-den} and further partitioning with respect to $ord_{\widetilde{V}}$ we have
$$\pi^{-1}(A) = \displaystyle\coprod_{i \in S, \, m \geq 0} A_{i,m}$$ where $A_{i,m} = U_{i,strict}^k$ denotes the set of arcs $\gamma \in Y_{\infty}$ such that $ord_{E_i}(\gamma) = 1$ and $ord_{\widetilde{V}} (\gamma) = m$.
Then, to compute this integral we can simply follow the classic approach of \cite[Lemma 3.3.3, Theorem 3.3.4]{greenbook} applying them to solely divisors of order $1$ and the strict transform (namely, the divisors $E_i + \widetilde{V}$ for $i \in S$).  
We have that
$$\mu(A_{i,0})= \left(\frac{1}{\Li} - \frac{1}{\Li^2}\right)\Li^{-(n-1)} [(E_i \setminus \widetilde{V})^o]$$
since following notations of \cite[Lemma 3.3.3]{greenbook} the set $A_{i,0}$ is the inverse image of $(\pi^2_1)^{-1}(0) \setminus (\pi^2_2)^{-1}(0)$ in the locally closed subset of 2-jets defined by $x_i=0$,
and for $m\geq 1$ 
$$\mu(A_{i,m})= \left(\frac{1}{\Li} - \frac{1}{\Li^2}\right) \Li^{-(n-1)-m} [(E_i \cap \widetilde{V})^o](\Li -1)$$
since the set $A_{i,m}$ is the inverse image of $(\pi^2_1)^{-1}(0) \setminus (\pi^2_2)^{-1}(0)$ in the locally closed subset of 2-jets defined by $x_i=0, x_{\widetilde{V}}=0$.
Therefore, the motivic integral $(\star)$ becomes
$$ (\star) = \displaystyle\sum_{(i,m) \in S \times \mathbb N} \mu(A_{i,m}) \Li^{-(N_i + m)s - \nu_i}
 = $$ 
 $$\left(\frac{1}{\Li} - \frac{1}{\Li^2}\right) \Li^{-(n-1)}\displaystyle\sum_{i \in S} 
 \left( \frac{1}{\Li^{N_i s + \nu_i}}[(E_i \setminus \widetilde{V})^o] + [(E_i \cap \widetilde{V})^o](\Li -1)\displaystyle\sum_{m \geq 1} \Li^{-(N_i + m)s -\nu_i} \Li^{-m} \right) =
 $$ 
 $$\left(\frac{1}{\Li} - \frac{1}{\Li^2}\right) \Li^{-(n-1)}\displaystyle\sum_{i \in S} \frac{1}{\Li^{N_i s + \nu_i}}
 \left( [(E_i \setminus \widetilde{V})^o] + [(E_i \cap \widetilde{V})^o](\Li -1)\displaystyle\sum_{m \geq 1} \Li^{-m(s +1)} \right) =
 $$
 $$
 \left(\frac{1}{\Li} - \frac{1}{\Li^2}\right)\Li^{-(n-1)} \displaystyle\sum_{i \in S} \frac{1}{\Li^{N_i s + \nu_i}} \left(  [(E_i \setminus \widetilde{V})^o]   +   \frac{1}{\Li^{s+1}-1} [(E_i \cap \widetilde{V})^o](\Li -1)   \right) .$$
\end{proof}

Combining Theorem \ref{thm-igusa-Q} and Theorem \ref{thm-denef} we can also compare the series of motivic classes of curvilinear Hilbert schemes directly with motivic classes of divisors. 

\begin{crl}
We have the relationship 
$$ (1 - \Li^s) Q_V( \Li^{-s - (n-1)} ) =$$
$$
\Li^{-(n-1)}\displaystyle\sum_{i \in S}  \left( \Li^{-N_i s}\Li^{-\nu_i}\left(   [(E_i \setminus \widetilde{V})^o]   +  \frac{1}{\Li^{s+1} -1 } [(E_i \cap \widetilde{V})^o] (\Li - 1)  \right)  \right)     +     \Li^{-s} \frac{\Li^{-3n+2} - \Li^{-2n+2}}{\Li -1} . $$ 
\end{crl}

In particular, the equality of the corollary above is an equality of formal power series in $\Li^{-s}$. Unraveling this equality and explicating the geometric series for $\displaystyle\frac{1}{\Li^{s+1} -1}$, the equality descends to the coefficients of fixed degree terms, obtaining a recursive relationship that computes all the motivic classes of curvilinear Hilbert schemes in terms of motivic classes of divisors. 

\begin{crl} \label{rec}
For $k \geq 2$ we have 
$$\Li^{-(n-1)k} [\CHilb^k_0(V)] - \Li^{-(k+1)(n-1)}[\CHilb^{k+1}_0(V)]) = $$
    $$
\Li^{-(n-1)} \left( 
\displaystyle\sum_{i \in S \, | \, N_i = k} 
\Li^{-\nu_i}[(E_i \setminus \widetilde{V})^o] \, \,
+ \, \,
\displaystyle\sum_{i \in S , \, m  \geq 1 \, | \, N_i +m = k } 
\Li^{ -\nu_i -m}(\Li - 1)[(E_i \cap \widetilde{V})^o] \right) .
$$
\end{crl}

\begin{rmk} \label{wei}
Having an equality of power series in motivic classes leading to an explicit description of $[\CHilb^k_0(V)]$, we can apply morphisms of the Grothendieck ring of varieties to them and the relationship of Corollary \ref{rec} to obtain more specialized information about $\CHilb^k_0(V)$. An interesting example can consider weight polynomials $$ X \mapsto \omega(X) = \omega(X) (t) = \sum_{j,k \geq 0} (-1)^{j+k} t^k \dimm \operatorname{Gr}_W^k (\operatorname{H}^j_c(X))$$ as defined in \cite{hodge}, and combining the results of Theorem \ref{thm-igusa-Q} and Theorem \ref{thm-denef} we have the possibility of expressing weight polynomials $\omega(\CHilb^k_0(X))$ of curvilinear Hilbert schemes in terms of the geometry of divisors. 
We are especially interested in the generating series of weight polynomials in the case of plane curve singularities because of the conjecture presented by Oblomkov, Rasmussen and Shende in \cite{ors} trying to give to the series of weight polynomials of Hilbert schemes an interpretation in terms of the topology of the singularity.
\end{rmk}

We now move on to a singular class of examples, that includes the cusp singularity of Example \ref{cus2}.
\begin{ex}
We consider the plane curve defined by $f(x,y)=x^2 - y^{2k+1}$.
Computing an embedded resolution graph, it looks like
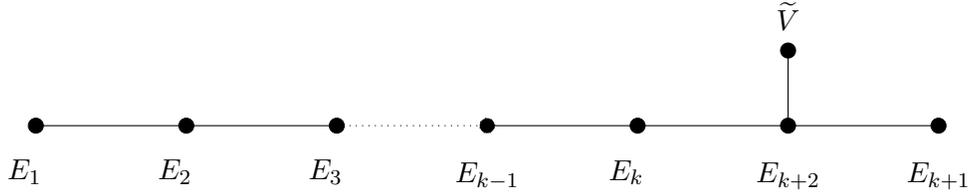
\vspace{5mm}
\begin{center}
\begin{tikzpicture}
\tikzstyle{every node}=[draw,shape=circle,fill,inner sep=2pt]
\draw (0,0) node[label=below:$E_{1\phantom{11}}$]{} -- (2,0) node[label=below:$E_{2\phantom{11}}$]{};
\draw (2,0) -- (4,0) node[label=below:$E_{3\phantom{11}}$]{};
\draw[dotted] (4,0) -- (6,0) node[label=below:$E_{k-1}$]{};
\draw (6,0) -- (8,0) node[label=below:$E_{k\phantom{11}}$]{};
\draw (8,0) -- (10,0) node[label=below:$E_{k+2}$]{};
\draw (10,0) -- (10,1) node[label=above:$\widetilde{V}$]{};
\draw (10,0) -- (12,0) node[label=below:$E_{k+1}$]{};
\end{tikzpicture}
\end{center}
and the resolution data is given by 
\begin{center}
    $(N_j, \nu_j) = (2j, j)$ for $j=1 \dots k$\\
    $(N_{k+1}, \nu_{k+1}) = (2k+1 , k+1)$.
\end{center}
We discard $E_{k+2}$ since $m_{k+2}=2$.
The corresponding Igusa zeta function can computed from the resolution using Theorem \ref{thm-denef}, just observing that
\begin{center}
    $[(E_i \setminus \widetilde{V})^o] = \Li -1$ for $j=2, \dots , k$\\
    $[(E_i \setminus \widetilde{V})^o] = \Li $ for $j=2, 2k+1$.
\end{center}
Then, to eventually compute the series $Q_C(T)$ one can directly check that $H_2 = \Li+1$, $H_3 = \Li^2$ and since there are no divisors of odd multiplicity from the recursive formula of Corollary \ref{rec} we have 
\begin{center}
    $H_{2j} = \Li H_{2j-1}$\\
    $H_{2j+1} = H_{2j} = \Li^j$ for $j=2 \dots k$
\end{center}
when we stop before becoming empty. 
\end{ex}

To conclude this section, we compute and compare the motivic series of Theorem \ref{thm-denef} on an example.
\begin{ex}
In the same setting as Example \ref{ex-cusp} the only divisors of order $1$ are $E_1$ and $E_2$. Using the recursive formula of Corollary \ref{rec} and computing manually $[\CHilb^2_0(C)] = \Li +1$ we obtain $[\CHilb^3_0(C)] = \Li$ and $[\CHilb^k_0(C)] = 0 \text{ for all } k \geq 4 $ thus we have $$\displaystyle\int_A (\Li^{-ord_C})^s d\mu = 
    \left(\frac{1}{\Li} - \frac{1}{\Li^2}\right)
    \left(
    \displaystyle\frac{1}{\Li^{2 s + 1}}
    +
    \displaystyle\frac{1}{\Li^{3 s + 2}}
    \right) .$$
\end{ex}

\section{Application: plane curve singularities} \label{ex}

This section will be dedicated to further exploring the objects and results of Section \ref{curv} in the case of plane curve singularities. Unless otherwise stated, $(C,0)$ will always denote a complex plane curve defined as the zero locus of a polynomial $f \in \mathbb C[x,y]$, $f(0)=0$ and sometimes additionally an isolated reduced singularity. 
We denote by $f_d$ the $d$-th degree homogeneous component of $f$, and the smallest $d >0$ such that $f_d \neq 0$ is called multiplicity of the singularity and denoted by $mult_C(0)$.

We are naturally interested in studying the geometric properties of plane curves, corresponding to studying the data coming from its resolution of singularities as presented in Section \ref{resol}, as well as the topological properties of its algebraic link introduced by Milnor in \cite{milnor} - which turn out to be equivalent to each other as proved in \cite[Chapter III.8]{Brieskorn}.

\begin{df} \label{def-link}
Let $(C, 0)$ be an isolated reduced complex plane curve singularity.
A sphere of radius $\epsilon$ centered at the singularity and the given curve intersect transversally for $\epsilon$ small enough, therefore we can consider the nonsingular oriented one-dimensional submanifold  
$$\mathcal{L} = S^3_{\epsilon , 0} \cap C $$
that is called the algebraic link of the singularity.
\end{df}

\begin{thm} \label{thm-bries}
Given an isolated complex plane curve singularity, the following data are equivalent and fully determine the singularity:
\begin{itemize}
    \item Its algebraic link characterized as iterated torus knots and linking numbers;
    \item The Puiseux pairs and intersection numbers of the branches of the curve;
    \item The resolution graph decorated with multiplicities.
\end{itemize}
\end{thm}

Furthermore, in Theorem \ref{thm-denef} we presented a relationship between motivic classes of curvilinear Hilbert schemes and the geometric data on an embedded resolution of singularities i.e. the divisors appearing from the resolution process. In the case of complex plane curves, this also guarantees that $[\CHilb^k_0(C)]$ is a polynomial in $\Li$ and because of the equivalence between the geometric and the topological data just mentioned for plane curves, the polynomials $[\CHilb^k_0(C)]$ (and consequently their weight polynomials, see Section \ref{val}) are expected to have a topological meaning for isolated reduced singularities.
\begin{crl}
Let $(C, 0)$ be a complex plane curve.
For any $k>1$ we have that $[\CHilb^k_0(C)]$ is polynomial in $\Li$ and assuming $(C,0)$ is an isolated reduced complex plane curve singularity, the polynomial $[\CHilb^k_0(C)]$ is a topological invariant of $(C,0)$.
\end{crl}

\begin{proof}
From Corollary \ref{rec} we recall that $$\Li^{-(n-1)k} [\CHilb^k_0(C)] - \Li^{-(k+1)(n-1)}[\CHilb^{k+1}_0(C)]) = $$
$$
\Li^{-(n-1)} \left( 
\displaystyle\sum_{i \in S \, | \, N_i = k} 
\Li^{-\nu_i}[(E_i \setminus \widetilde{V})^o] \, \,
+ \, \,
\displaystyle\sum_{i \in S , \, m  \geq 1 \, | \, N_i +m = k } 
\Li^{ -\nu_i -m}(\Li - 1)[(E_i \cap \widetilde{V})^o] \right)
$$ and we can prove the polynomiality in $\Li$ of $[\CHilb^k_0(C)]$ by proving it for all the motivic classes involved in the recursive formula.
We start with $[\CHilb^2_0(C)] = \displaystyle\frac{[S^2_0(C)]}{\Li - 1}$. In $S^2_0(C)$ we have the 2-jets $(at, bt) \in (\C[t]/(t^2))^2$ such that $(a,b) \neq (0,0)$ and such that $f(at,bt) = 0$ in $\C[t]/(t^2)$.
If $C$ is smooth then $S^2_0(C) \cong \mathbb \{ 0\} \times A^1  \setminus \{ 0\}$ hence $$[\CHilb^2_0(C)] = \frac{\Li -1}{\Li - 1} = 1$$ and if $C$ is singular then $S^2_0(C) \cong \mathbb A^2 \setminus \{ (0,0)\}$ hence $$[\CHilb^2_0(C)] = \frac{\Li^2 -1}{\Li - 1} = \Li  +1.$$
Then, we move on to divisors. Since they can only be $\mathbb P^1$, $(E_i \setminus \widetilde{V})^o$ is a $\mathbb P^1$ with a finite number $k_i$ of points removed, thus $$[(E_i \setminus \widetilde{V})^o] = (\Li +1) - k_i.$$ Then, since they are transversal, the intersection $E_i \cap \widetilde{V}$ corresponds to a finite number of points $s_i$ and $$[(E_i \cap \widetilde{V})^o] = s_i.$$
Finally, the equivalence between the geometric and the topological data comes from Theorem \ref{thm-bries}.
\end{proof}

This expectation is also partially met in a conjecture formulated by Oblomkov, Rasmussen and Shende in \cite{ors} but not in the curvilinear setting. The conjecture arises from another possible approach to study links, that makes use of homological invariant theories such as the triply-graded link homology introduced by Khovanov and Rozansky in \cite{kr1,kr2} and polynomial invariants built from these homology theories. 
In this different setting, (weight polynomials of) Hilbert schemes at the curve singularity come again into play offering a fruitful geometric framework to give link invariants a geometric interpretation. In this direction we can also find specialized theorems for other classical polynomial invariants by Campillo, Delgado and Gusein-Zade in \cite{cdgz} and by Maulik in \cite{mau}.

\subsection{Properties of the curvilinear Hilbert scheme of plane curve singularities}

We conclude by studying some special properties of the curvilinear Hilbert schemes of plane curve singularities and computing some examples. 

We can say from the relationship presented in Theorem \ref{thm-iso2} that $ I = (g_1 , \dots , g_n) = \operatorname{ker} (\phi_I) \in \CHilb^k_0(C)$
for $\phi_I: \mathbb C [x,y] \to \mathbb C[t] / (t^k)$ the parametrization of the smooth curve $g_i$.
As a consequence, when imposing $f \in I$ the parametrization fails to describe the ideal unless $f \in m^n$.  
This forces the curvilinear component of the Hilbert scheme to be empty and an exact threshold when this happens is obtained next.

In particular, this result ensures that the curvilinear series $Q_C(T)$ of Definition \ref{def-ser} is a polynomial without relying on the formula proved in Theorem \ref{thm-denef}.
Vice versa, the same threshold can be obtained as the top degree of the motivic description presented in Theorem \ref{thm-denef}. 

\begin{thm} \label{thm-emp}
Let $(C, 0)$ be a complex plane curve singularity with no smooth branches and let $$N = \max \{ N_i \, | \, i \in S \}$$ for an embedded resolution as in Theorem \ref{thm-denef}. Then, $\CHilb^k_0(C)$ is empty for any $k>N$ and this bound is optimal.
\end{thm}

\begin{rmk}
We just mention that the stabilization-to-empty threshold admits as well a formulation and consequent proof in terms of the Puiseux expansion of the curve, which we omit for the sake of clarity.
\end{rmk}

\begin{proof}
First of all, if $C = \displaystyle\cup_i C_i$ with $C_i$ an irreducible curve, then $\CHilb^k_0(C_i) \subseteq \CHilb^k_0(C)$ for all $k \geq 1$. For this reason, we need to avoid smooth branches since for $C_i$ smooth we would have $\emptyset \neq \Hilb^k_0(C_i) = \CHilb^k_0(C_i) \subseteq \CHilb^k_0(C)$ for all $k \geq 1$.
    
Then, we can reduce ourselves to work with an irreducible singular curve and we observe the following. 

\begin{lemma}
  Let $\pi: (Y, E) \rightarrow (\mathbb A^2,C)$ be an embedded resolution of the pair $(\mathbb A^2,C)$ as in Definition \ref{resol}. The strict transform of $C$ can never intersect a divisor of order $1$. 
\end{lemma}

\begin{proof}
  We proceed by contradiction, assuming there is a point $p \in E_i \cap \widetilde{C}$ for some divisor $E_i$ with $m_i = 1$.
  Then, we can consider an arc $\gamma \in \widetilde{C}_\infty$ such that $ord_{E_i}(\gamma) = 1$. 
  From Proposition \ref{thm-val}, it corresponds to an arc of $A \cap C_\infty$ which is empty hence leading to a contradiction since for any $\eta \in (\mathbb A^2)_\infty$ we have $ord_C(\eta) \geq mult_C(0) > 1$.
\end{proof}

We can now conclude our proof. From the formula of Corollary \ref{rec} we immediately observe that for $k > N$ we do not have any contribution from any of the two factors $[(E_i \setminus \widetilde{V})^o]$ and $[(E_i \cap \widetilde{V})^o]$, and this bound is also optimal since for $k=N$ we have the contribution of the corresponding divisor $E_i$.
\end{proof}

\begin{rmk}
We need to mention that when considering varieties of higher dimension as in Section \ref{curv}, we lose the stabilization-to-empty results for curvilinear schemes of Theorem \ref{thm-emp} as it relies on the $n=2$ assumption (for example there are parameterizable, smooth curves passing through a surface singularity).
Moreover, we lose Theorem \ref{thm-bries} providing equivalence between topology and embedded resolutions.
\end{rmk}

To conclude this section, we check the bound we just proved on a classical example.
\begin{ex} \label{cus2}
    In Example \ref{ex-cusp} we computed the resolution and corresponding multiplicities of the cusp singularity, and the corresponding order are $$(N_i, \nu_i, m_i)_{i=1,2,3} = ((2, 1, 1),(3, 2, 1),(6, 4, 2))$$ hence the bound is then given by $N=3$. 
    Choosing coordinates where $f = y^2-x^3$ and choosing representatives for the points of the Hilbert scheme we have
        $$\Hilb^2_0(C) = \{ (x+ \lambda y)_{\lambda \in \C} , (y, x^2)\} = \CHilb^2_0(C)$$ 
        $$\Hilb^3_0(C) = \{ (y+ \lambda x^2 )_{\lambda \in \C} ,(x^2, xy )\}, \;\CHilb^3_0(C) = \{ (y+ \lambda x^2 )_{\lambda \in \C} \}$$
        $$\Hilb^4_0(C) = \{ (x^2 + \lambda xy)_{\lambda \in \C} , (xy, x^3)\}, \; \CHilb^4_0(C) =\emptyset$$ 
        $$\Hilb^5_0(C) = \{ (xy+ \lambda x^3)_{\lambda \in \C} , (x^3, x^2y)\}, \; \CHilb^5_0(C) = \emptyset$$ 
        $$\dots$$
    and we can see that the curvilinear components stabilize to empty for $k > 3$ and the bound is optimal.
\end{ex}

\section{References}
\printbibliography[heading=none]

\end{document}